\documentclass[12pt]{amsart}
\usepackage{hyperref} 
\usepackage{amsmath, amsthm, amssymb}
\usepackage{hyperref} 
\usepackage{enumerate}
\usepackage{verbatim}
\usepackage{caption}
\usepackage{subcaption}
\usepackage{esint}
\usepackage[T1]{fontenc}
\usepackage[doi=false,isbn=false,url=false]{biblatex} 
\usepackage{tikz,amsthm,amsmath,amstext,amssymb,amscd,epsfig,euscript, mathrsfs,
 dsfont,pspicture,
multicol,graphpap,graphics,graphicx,times,enumerate,
sidecap,
wrapfig,color,pict2e}
\usepackage{setspace}

\numberwithin{equation}{section}

\title{Quantitative differentiation and the medial axis}
\date{\today}
\author{Guy C. David}
\address{Department of Mathematical Sciences\\ Ball State University, Muncie, IN 47306}
\email{gcdavid@bsu.edu}
\author{Kevin Hook}
\address{Department of Mathematical Sciences\\ Ball State University, Muncie, IN 47306}
\email{kmhook2@bsu.edu}

\thanks{G.~ C.~ David was partially supported by the National Science Foundation under Grant No. DMS-2054004. The results of this paper form part of the 2022 master's thesis of K.~Hook at Ball State University.}
\subjclass[2020]{28A75}

\begin{document}

\theoremstyle{plain}
\newtheorem{theorem}{Theorem}
\newtheorem{exercise}{Exercise}
\newtheorem{corollary}[theorem]{Corollary}
\newtheorem{scholium}[theorem]{Scholium}
\newtheorem{claim}[theorem]{Claim}
\newtheorem{lemma}[theorem]{Lemma}
\newtheorem{sublemma}[theorem]{Lemma}
\newtheorem{proposition}[theorem]{Proposition}
\newtheorem{conjecture}[theorem]{Conjecture}

\theoremstyle{definition}
\newtheorem{fact}[theorem]{Fact}
\newtheorem{example}[theorem]{Example}
\newtheorem{definition}[theorem]{Definition}
\newtheorem{remark}[theorem]{Remark}
\newtheorem{question}[theorem]{Question}

\numberwithin{equation}{section}
\numberwithin{theorem}{section}

\newcommand{\cG}{\mathcal{G}}
\newcommand{\RR}{\mathbb{R}}
\newcommand{\HH}{\mathcal{H}}
\newcommand{\LIP}{\textnormal{LIP}}
\newcommand{\Lip}{\textnormal{Lip}}
\newcommand{\Tan}{\textnormal{Tan}}
\newcommand{\length}{\textnormal{length}}
\newcommand{\dist}{\textnormal{dist}}
\newcommand{\diam}{\textnormal{diam}}
\newcommand{\vol}{\textnormal{vol}}
\newcommand{\rad}{\textnormal{rad}}
\newcommand{\side}{\textnormal{side}}

\def\bA{{\mathbb{A}}}
\def\bB{{\mathbb{B}}}
\def\bC{{\mathbb{C}}}
\def\bD{{\mathbb{D}}}
\def\bR{{\mathbb{R}}}
\def\bS{{\mathbb{S}}}
\def\bO{{\mathbb{O}}}
\def\bE{{\mathbb{E}}}
\def\bF{{\mathbb{F}}}
\def\bH{{\mathbb{H}}}
\def\bI{{\mathbb{I}}}
\def\bT{{\mathbb{T}}}
\def\bZ{{\mathbb{Z}}}
\def\bX{{\mathbb{X}}}
\def\bP{{\mathbb{P}}}
\def\bN{{\mathbb{N}}}
\def\bQ{{\mathbb{Q}}}
\def\bK{{\mathbb{K}}}
\def\bG{{\mathbb{G}}}

\def\nrj{{\mathcal{E}}}
\def\cA{{\mathscr{A}}}
\def\cB{{\mathscr{B}}}
\def\cC{{\mathscr{C}}}
\def\cD{{\mathscr{D}}}
\def\cE{{\mathscr{E}}}
\def\cF{{\mathscr{F}}}
\def\cB{{\mathscr{G}}}
\def\cH{{\mathscr{H}}}
\def\cI{{\mathscr{I}}}
\def\cJ{{\mathscr{J}}}
\def\cK{{\mathscr{K}}}
\def\Layer{{\rm Layer}}
\def\cM{{\mathscr{M}}}
\def\cN{{\mathscr{N}}}
\def\cO{{\mathscr{O}}}
\def\cP{{\mathscr{P}}}
\def\cQ{{\mathscr{Q}}}
\def\cR{{\mathscr{R}}}
\def\cS{{\mathscr{S}}}
\def\Up{{\rm Up}}
\def\cU{{\mathscr{U}}}
\def\cV{{\mathscr{V}}}
\def\cW{{\mathscr{W}}}
\def\cX{{\mathscr{X}}}
\def\cY{{\mathscr{Y}}}
\def\cZ{{\mathscr{Z}}}

  \def\del{\partial}
  \def\diam{{\rm diam}}
	\def\VV{{\mathcal{V}}}
	\def\FF{{\mathcal{F}}}
	\def\QQ{{\mathcal{Q}}}
	\def\BB{{\mathcal{B}}}
	\def\XX{{\mathcal{X}}}
	\def\PP{{\mathcal{P}}}

  \def\del{\partial}
  \def\diam{{\rm diam}}
	\def\image{{\rm Image}}
	\def\domain{{\rm Domain}}
  \def\dist{{\rm dist}}
	\newcommand{\Gr}{\mathbf{Gr}}
\newcommand{\md}{\textnormal{md}}
\newcommand{\vspan}{\textnormal{span}}

\begin{abstract}
We study the medial axis of a set $K$ in Euclidean space (the set of points in space with more than one closest point in $K$) from a ``coarse'' and ``quantitative'' perspective. We show that on ``most'' balls $B(x,r)$ in the complement of $K$, the set of almost-closest points to $x$ in $K$ takes up a small angle as seen from $x$. In other words, most locations and scales in the complement of $K$ ``appear'' to fall outside the medial axis if one looks with only a certain finite resolution. The word ``most'' involves a Carleson packing condition, and our bounds are independent of the set $K$.
\end{abstract}

\maketitle

\section{Introduction}


If $K\subseteq \RR^k$, the distance from a point $p\in\RR^k$ to the set $K$ is
$$ d(p,K) = \inf\{d(p,x) : x\in K\},$$
where $d(p,x)$ denotes the Euclidean distance $|p-x|$. If $K$ is closed and $p\in \RR^k$, then there is always a point $x\in K$ such that $d(p,x)=d(p,K)$, but this point $x$ may not be unique.

The set of points $p\in\mathbb{R}^k$ for which this closest point $x\in K$ is \textit{not} unique is called the \textit{medial axis} of $K$, which we denote $\text{Med}(K)$:
\begin{definition}
Given $K\subseteq \RR^k$, let
\begin{equation*}
    \text{Med}(K)=\{p\in \mathbb{R}^k: \text{ there exist $x,y\in K$ with $x\neq y$ and $d(p,K)=d(p,x)=d(p,y)$\}}.
\end{equation*}
\end{definition}
The medial axis has a fairly long history in both pure and applied mathematics; the results of Erd\"os \cite{erdos} appear before even the name ``medial axis'' was coined by Blum \cite{Blum}. A good overview from the pure mathematical perspective can be found in the introduction and references of \cite{hajlasz}.

For one thing, it is well-known that the medial axis of any closed set $K$ has measure zero. A short proof of this fact can be given by applying Rademacher's theorem on the differentiability of Lipschitz functions to the Lipschitz function $x\mapsto\text{dist}(x,K)$. See \cite{mathoverflow} or \cite[Remark 13]{hajlasz} for details. In fact, much stronger results than this can be proven on the smallness of the medial axis: see \cite{erdos, Fremlin}.

In this paper, we prove that the medial axis is small from a ``coarse'' or ``quantitative'' perspective. In essence, our main result (Theorem \ref{thm:main} below) says that given a compact set $K\subseteq\RR^k$, the set of locations and scales in the complement of $K$ that ``appear'' to be in the medial axis is ``small'', with a control which is independent of the set $K$.  The words ``small'' and ``appear'' need some further elaboration.

To measure the size of a collection of locations and scales in $\mathbb{R}^k$, we use the notion of a Carleson set:
\begin{definition}\label{c-set r^k intro}
Let $D\subseteq \mathbb{R}^k\times \mathbb{R}^+$ be measurable. Let $D_r=\{x:(x,r)\in D\}.$ We say that $D$ is a \textit{Carleson set} if there is a $C\geq 0$ such that for every $L>0$ and every ball $B\subseteq \mathbb{R}^k$ of radius $L$,
$$\int_0^L|D_r\cap B|\frac{dr}{r}\leq C|B|,$$
where $|\cdot|$ denotes the Lebesgue measure of a set in $\RR^k$. We call the minimal $C$ for which this is satisfied the \textit{Carleson constant} of $D$.
\end{definition}

This definition plays a major role in the area of quantitative geometric measure theory developed by David and Semmes \cite{davidsemmes}. Roughly speaking, if one thinks of $D\subseteq \mathbb{R}^k\times \mathbb{R}^+$ as a collection of balls in $\mathbb{R}^k$ (centers and radii), the Carleson condition says that this is a ``small'' collection: most points of $\RR^k$ are not contained in too many balls of $D$ of very different radii. In particular, if $D$ is Carleson then every ball in $\RR^k$ contains a ball of comparable size lying outside the collection $D$. A nice discussion of this concept is given by Semmes in \cite[B.29]{gromov}.

Our main theorem considers the set of all balls $B(x,r)$ in the complement of a given set $K\subseteq\RR^k$. Some of these balls may have the property that there are two points $z_1$ and $z_2$ in $K$ such that
\begin{itemize}
    \item $z_1$ and $z_2$  both ``almost'' minimize the distance to $x$ in $K$, i.e.,
    $$ d(x,z_i) \leq d(x,K) + \epsilon r $$
    for some small $\epsilon>0$, and
    \item the angle between the segments $[x,z_1]$ and $[x,z_2]$ is ``large'', i.e., bounded away from zero in some quantitative way. 
\end{itemize}
We think of a ball satisfying these conditions as representing a point that ``appears'' to be in the medial axis if one looks with only some finite degree of resolution.

Our main theorem says that such balls are rare: they form a Carleson set. Moreover, the Carleson constant $C$ is completely independent of the set $K$, depending only on the dimension and the chosen parameters.

\begin{theorem}\label{thm:main}
Let $K\subseteq \mathbb{R}^k$ and let $\epsilon\geq 0$ and $\delta>0$ satisfy $2\delta+\epsilon<1$. Let
\begin{align*}
G=& \{(x,r):x\in \mathbb{R}^k, 0<r<d(x,K),\text{ and there exist }z_1,z_2\in K \text{ such that }\\
&d(x,z_1),d(x,z_2)\leq d(x,K)+\epsilon r\\& \text{ but }|\theta|>\cos^{-1}\left(2\left(\frac{1-(2\delta +\epsilon)}{1+2\delta}\right)^2-1\right)\}
\end{align*}
where $\theta$ is the angle between $[x,z_1]$ and $[x,z_2]$. 
Then $G$ is a Carleson set whose Carleson constant can be bounded above depending only on $\delta$ and $k$. 
\end{theorem}
Note that the quantity $\cos^{-1}\left(2\left(\frac{1-(2\delta +\epsilon)}{1+2\delta}\right)^2-1\right)$ is positive if $\delta>0$ and tends to $0$ as $\delta,\epsilon\rightarrow 0$. The theorem is already of interest in the case $\epsilon=0$, in which case it more directly concerns the medial axis of $K$.

Thus, Theorem \ref{thm:main} gives a precise sense in which every medial axis is seen only at a small collection of locations and scales, in a way which is robust and independent of the base set $K$.

The proof of Theorem \ref{thm:main} uses the philosophy of the proof mentioned above that the medial axis has measure zero, which applies Rademacher's theorem to the distance function to $K$ \cite[Remark 13]{hajlasz}. However, instead of using Rademacher's theorem, which provides only infinitesimal and not ``coarse'' information, we use a result from the theory of ``quantitative differentiation'', explained in Section \ref{sec:background}. In Section \ref{sec:proof} we apply this theorem along with some quantitative estimates on the distance function to prove Theorem \ref{thm:main}.

To conclude the introduction, we emphasize that Theorem \ref{thm:main} is not implied by the fact that the medial axis has measure zero, nor even by the stronger results of \cite{erdos,Fremlin} mentioned earlier, because those results make no statements about the ``large scale'' structure of the medial axis.

\section{Background on quantitative differentiation}\label{sec:background}
The theory of ``quantitative differentiation'' considers approximation by Lipschitz functions at ``large'' scales, not just infinitesimal ones. It originates in work of Dorronsoro \cite{dorronsoro} and Jones \cite{Jones}, with extensions by many others since. Good overviews of the material we need can be found in Appendix B by Semmes in \cite{gromov} (especially Section B.29), the notes of Young \cite{young}, or the master's thesis of the second named author \cite{hook}.

With the dimension $k$ understood from context, let $B(x,r)$ denote the closed ball of radius $r$ centered at $x\in\RR^k$. We use $\mathbb{R}^+$ below to denote the positive real numbers.

\begin{definition}\label{ ecd r^k}
Let $f:\mathbb{R}^k\rightarrow \mathbb{R} \text{ and } \epsilon >0$. We say that  $f$ is $\epsilon$-coarsely-differentiable on $B(x,r)$ if there is an affine function $\lambda:\mathbb{R}^k\rightarrow \mathbb{R}$ such that $$|f(p)-\lambda(p)|\leq \epsilon r \text { for all } p\in B(x,r).$$
\end{definition}

The main result we need is that every $1$-Lipschitz function is coarsely differentiable on all balls outside of a Carleson set. 

\begin{theorem}\label{final qd intro}
Let $\epsilon>0$. Let $f:\mathbb{R}^k\rightarrow \mathbb{R}$ be 1-Lipschitz. Let
$$G=\{(x,r)\in\RR^k\times\RR^+: f \text { is not $\epsilon$-coarsely differentiable on } B(x,r)\}.$$
Then $G$ is Carleson, and its Carleson constant can be bounded above depending only on $\epsilon$ and $k$.
\end{theorem}

As Semmes remarks in \cite[B.29]{gromov}, it is difficult to trace the attribution of this precise result. It follows from the main results of \cite{dorronsoro}, as discussed in \cite{davidsemmes}. It is stated explicitly as \cite[Theorem B.29.10]{gromov} and as \cite[Theorem 2.4]{young}. An exposition of the proof (following that of Young \cite{young}) is given in \cite{hook}, and a generalization to metric space targets is given in \cite{AzzamSchul}. Further generalizations and analogs are discussed in \cite{cheeger}.

\section{Proof of the main theorem}\label{sec:proof}
The main work in the proof of Theorem \ref{thm:main} lies in the following result.

\begin{theorem}\label{ z_1,z_2 close together e}
Let $K\subseteq \mathbb{R}^k$ be compact. Let $f:\mathbb{R}^k\rightarrow \mathbb{R}$ be $f(x)=d(x,K)$. Let $\epsilon\geq 0$ and $\delta>0$ be such that $2\delta+\epsilon<1$. Let $x\in \mathbb{R}^k$ be such that there are $z_1,z_2\in K$, where
\begin{equation}\label{d(x,z_1)< d(x,K)+er}
    d(x,z_1)\leq d(x,K)+\epsilon r,
\end{equation}
$$d(x,z_2)\leq d(x,K)+\epsilon r, $$
and $0<r<d(x,K)$.
Suppose that $f$ is $\delta$-coarsely differentiable on $B(x,r)$.  Let $\theta$ be the angle between $[x,z_1]$ and $[x,z_2]$. Then
$$|\theta|\leq \cos^{-1}\left(2\left(\frac{1-(2\delta+\epsilon)}{1+2\delta}\right)^2-1\right).$$
\end{theorem}
This theorem says that if $f$ is $\delta$-coarsely differentiable on some ball $B(x,r)$, then the set of ``almost-closest'' points to $x$ in $K$ must occupy a small angle as seen from $x$.

To prove this, we first need a few lemmas. 
\begin{lemma}\label{f(x)=d(x,K)}
Let $K\subseteq\mathbb{R}^k$. Let $f:\mathbb{R}^k\rightarrow \mathbb{R}$ be $f(x)=d(x,K)$. Then $f$ is 1-Lipschitz.
\end{lemma}
\begin{proof}
This is a well-known consequence of the triangle inequality, so we omit the simple proof.
\end{proof}

\begin{lemma} \label{ extended f(p)-f(y)}
Let $K\subseteq\mathbb{R}^k$ and $f:\mathbb{R}^k\rightarrow \mathbb{R}$ be $f(x)=d(x,K)$. Suppose that $x\in \mathbb{R}^k$, $z\in K$, $\epsilon\geq 0$, and $r>0$ satisfy
$$d(x,z)\leq d(x,K)+\epsilon r.$$
Let $y\in [x,z]$. Then 
$$f(x)-f(y)\geq d(x,y)-\epsilon r.$$
\end{lemma}
\begin{proof}
It suffices to show that
$$f(y)\leq f(x)+\epsilon r -d(x,y).$$
As $x,y,z\in [x,z]$,
$$d(x,z)=d(x,y)+d(y,z),$$
and so
$$d(y,z)=d(x,z)-d(x,y).$$
Thus, 
\begin{align*}
    f(y)&=d(y,K)
    \\&\leq d(y,z)
    \\&=d(x,z)-d(x,y)
    \\&\leq d(x,K)+\epsilon r -d(x,y)
    \\&= f(x)+\epsilon r -d(x,y).
\end{align*}
\end{proof}
\begin{lemma}\label{1+2e}

Let $f\colon\RR^k\rightarrow\RR$ be $1$-Lipschitz. Assume that there is a ball $B(x,r)$ and an affine function $A\colon \RR^k\rightarrow\RR$ such that
\begin{equation}\label{f(t)-A(t)}
|f(t)-A(t)|\leq \delta  r \text{ for all } t\in B(x,r). 
\end{equation}
Write 
\begin{equation}\label{A(t)=L(t)+C)}
A(t)=L(t)+C,
\end{equation}
where $L$ is linear and $C\in\RR$.

Then for all vectors $v\in \mathbb{R}^k$, 
$$|L(v)|\leq (1+2\delta)||v||.$$
\end{lemma}
It suffices to prove this for all unit vectors $v\in \mathbb{R}^k$. To see why, assume the statement holds for all unit vectors and let $v\in \mathbb{R}^k$. Then as $L$ is linear
$$\frac{|L(v)|}{||v||}=L\left(\frac{v}{||v||}\right)\leq (1+2\delta). $$
Thus 
$$|L(v)|\leq (1+2\delta)||v||.$$
\begin{proof}
Let $v\in \mathbb{R}^k$ be a unit vector.
Thus 
\begin{align*}
|A(x)-A(x+rv)|&=|A(x)-f(x)|+|f(x)-f(x+rv)|+|f(x+rv)-A(x+rv)|
\\& \leq \delta r + r+\delta r 
\\&=r(1+2\delta ).
\end{align*}
But as $L$ is linear 
\begin{align*}
|A(x)-A(x+rv)|&=|L(x)+C-(L(x+rv)+C)|
\\&=|L(x)-L(x+rv)|
\\&=|L(rv)|
\\&=r|L(v)|.
\end{align*}
Thus 
$$r|L(v)|\leq r(1+2\delta)$$
so 
$$|L(v)|\leq 1+2\delta. $$
\end{proof}

\begin{proof}[Proof of Theorem  \ref{ z_1,z_2 close together e}]
Let 
$$D=\{p\in \mathbb{R}^k:d(p,x)=r\}.$$
Let $y_1\in D\cap [x,z_1]$ and $y_2\in D\cap [x,z_2]. $ Note that these points exist because
$$ d(x, z_i) \geq d(x,K) > r$$
by assumption.

Denote the vectors from $y_i$ to $x$ by
$$w_1=x-y_1\text{ and } w_2=x-y_2.$$

We wish  to find bounds  for $$ \left |L\left(\frac{w_1+w_2}{2}\right)\right|\text{ and } \left|\left|\frac{w_1+w_2}{2}\right|\right|.$$ 

By \eqref{d(x,z_1)< d(x,K)+er} and as $y_1\in [x,z_1]$, we can say by Lemma \ref{ extended f(p)-f(y)} that
\begin{align*}
    f(x)-f(y_1)&\geq d(x,y_1)-\epsilon r \\&
    =r-\epsilon r\\&=r(1-\epsilon )\\&>0.
\end{align*}
Thus 
\begin{equation}\label{ub bound f(y_1) e}
    f(y_1)\leq f(x)-r(1-\epsilon).
\end{equation}
Similarly, 
$$f(y_2)\leq f(x)-r(1-\epsilon).$$
As $f$ is 1-Lipschitz and by our work from above,
\begin{align*}
   f(x)-f(y_1)&= |f(x)-f(y_1)|
    \\& \leq d(x,y_1)
    \\& =r.
\end{align*}
So
\begin{equation}\label {lb f(y_1) e}
    f(y_1)\geq f(x)-r.
\end{equation}
Similarly,
$$f(y_2)\geq f(x)-r.$$ 
Let $A$ and $L$ be the same functions from Lemma \ref{1+2e}.
By  \eqref{f(t)-A(t)},
$$f(y_1)-\delta r\leq A(y_1)\leq f(y_1)+\delta r.$$
By \eqref{ub bound f(y_1) e}, and \eqref{lb f(y_1) e},
$$f(x)-r-\delta r\leq A(y_1)\leq f(x)-r(1-\epsilon)+\delta r.$$
By \eqref{A(t)=L(t)+C)},
\begin{equation}\label{u/l bound for L(y_1)}
f(x)-r-\delta r -C\leq L(y_1)\leq f(x)-r+r(\delta +\epsilon)-C.
\end{equation}
Similarly 
$$f(x)-r-\delta r -C\leq L(y_2)\leq f(x)-r+r(\delta+\epsilon)-C.$$
By \eqref{f(t)-A(t)} and \eqref{A(t)=L(t)+C)}
$$-\delta r\leq f(x)-(L(x)+C)\leq \delta r.$$
Thus 
\begin{equation}\label{bound for L(x)}
f(x)-C -\delta r \leq L(x)\leq f(x)-C+\delta r.
\end{equation}

Thus as $L$ is linear and by \eqref{u/l bound for L(y_1)}, \eqref{bound for L(x)}
\begin{align*}
     L(w_1)&=L(x-y_1)
     \\&=L(x)-L(y_1)
     \\&\geq f(x)-C-\delta r -(f(x)-r+r(\epsilon+\delta)-C)
     \\& =-2\delta r -\epsilon r +r
     \\&= r(1-(2\delta +\epsilon)).
 \end{align*}
Similarly 
$$L(w_2)\geq r(1-(2\delta+\epsilon)).$$
Thus as $L$ is linear
\begin{equation}\label{lb for (w_1+w_2)/2 e}
\left|L\left(\frac{w_1+w_2}{2}\right)\right|\geq L\left(\frac{w_1+w_2}{2}\right)=\frac{L(w_1)+L(w_2)}{2}\geq r(1-(2\delta+\epsilon)).
\end{equation}
Now, to obtain a bound for $\left|\left|\frac{w_1+w_2}{2}\right|\right|^2$, note first that $||w_1||=||w_2||=r$. Thus,
\begin{align*}
    \left|\left|\frac{w_1+w_2}{2}\right|\right|^2&=\frac{1}{4}\left|\left|w_1+w_2\right|\right|^2
    \\&=\frac{1}{4}((w_1+w_2)\cdot (w_1+w_2))
    \\&=\frac{1}{4}(||w_1||^2+2(w_1\cdot w_2)+||w_2||^2)
    \\&=\frac{1}{4}(r^2+2(w_1\cdot w_2)+r^2)
    \\&=\frac{r^2}{2}+\frac{1}{2}(w_1\cdot w_2)
    \\&=\frac{r^2}{2}+\frac{1}{2}||w_1|| ||w_2||\cos(\theta)
    \\&=\frac{r^2}{2}+\frac{r^2}{2}\cos(\theta)
    \\&=r^2\left(\frac{1+\cos(\theta)}{2}\right).
\end{align*}
Thus,
\begin{equation}\label{length of (w_1+w_2)/2 e}
\left|\left|\frac{w_1+w_2}{2}\right|\right|=r\sqrt{\frac{1+\cos(\theta)}{2}}.
\end{equation}
Thus, by \eqref{lb for (w_1+w_2)/2 e}, \eqref{length of (w_1+w_2)/2 e} and Lemma \ref{1+2e},
\begin{align*}
r(1-(2\delta+\epsilon))&\leq \left|L\left(\frac{w_1+w_2}{2}\right)\right|
\\&\leq (1+2\delta)\left|\left|\frac{w_1+w_2}{2}\right|\right|
\\&\leq r(1+2\delta)\sqrt{\frac{1+\cos(\theta)}{2}}.
\end{align*}
Thus,
$$(1-(2\delta+\epsilon))\leq (1+2\delta)\sqrt{\frac{1+\cos(\theta)}{2}}.$$
Now solving for $\theta$
$$\left(2\left(\frac{1-(2\delta+\epsilon)}{1+2\delta}\right)^2-1\right)\leq 
\cos(\theta),$$
i.e.,
$$|\theta|\leq \cos^{-1}\left(2\left(\frac{1-(2\delta +\epsilon)}{1+2\delta}\right)^2-1\right).$$
\end{proof}

Now Theorem \ref{thm:main} follows directly from Theorem \ref{final qd intro} and Theorem \ref{ z_1,z_2 close together e}. 
\begin{proof}[Proof of Theorem \ref{thm:main}]
Let $f:\mathbb{R}^k\rightarrow \mathbb{R}$ be $f(x)=d(x,K)$. 

The measurability of the set $G\subseteq \RR^k\times\RR^+$ in Theorem \ref{thm:main} follows by standard arguments. Briefly, given $t>0$, let
\begin{align*}
G_t=&\{(x,r):x\in \mathbb{R}^k, 0<r<d(x,K),\text{ and there exist }z_1,z_2\in K \text{ such that }\\
&d(x,z_1),d(x,z_2)< d(x,K)+\epsilon r+t\\& \text{ but }|\theta|>\cos^{-1}\left(2\left(\frac{1-(2\delta +\epsilon)}{1+2\delta}\right)^2-1\right)\},
\end{align*}
where $\theta$ again denotes the angle of $[x,z_1]$ and $[x,z_2]$. Then each $G_t$ is open in $\RR^k\times\RR^+$ and
$$ G = \bigcap_{n=1}^\infty G_{1/n},$$
so is therefore measurable.

To see the Carleson condition, it follows from Theorem \ref{ z_1,z_2 close together e} that if $(x,r)\in G$ then $f$ is not $\delta$ -coarsely differentiable on $(x,r)$. Let $H$ denote the collection of $(x,r)$ such that $f$ is not $\delta$-coarsely differentiable on $B(x,r)$. 
Then $G\subseteq H$. By Lemma \ref{f(x)=d(x,K)}, $f$ is 1-Lipschitz. By Theorem \ref{final qd intro}, $H$ is Carleson, with Carleson constant bounded above depending only on $\delta$. Thus, $G$ is Carleson with Carleson constant bounded above by that of $H$.
\end{proof}
\printbibliography

\end{document}